\documentclass[twoside]{article}
\usepackage{amsmath,amssymb,amsthm,latexsym,bm,indentfirst,wasysym}
\usepackage{graphicx,epsfig,amscd,mathrsfs,multirow,bm}
\usepackage{cite}
\usepackage{caption}
\usepackage{color}
\usepackage[colorlinks,linkcolor=red,anchorcolor=red,citecolor=blue]{hyperref}
\usepackage{csquotes}
\numberwithin{equation}{section}
\usepackage{titlesec}
\newtheorem{Theorem}{Theorem}[section]
\newtheorem{Corollary}[Theorem]{Corollary}
\newtheorem{Lemma}[Theorem]{Lemma}
\newtheorem{Proposition}[Theorem]{Proposition}
\theoremstyle{Definition}
\newtheorem{definition}[Theorem]{Definition}
\theoremstyle{remark}
\newtheorem{Remark}[Theorem]{Remark}
\newtheorem*{Example}{Example}
\numberwithin{equation}{section}
\titleformat{\section}{\large\bfseries}{}{0pt}{}

\input xy
\xyoption{all}
\input scrload.tex

\allowdisplaybreaks[4]

\catcode`@=11
\long\def\@makefntext#1{\noindent #1}
\newskip\tabcentering \tabcentering=1000pt plus 1000pt minus 1000pt
\def\MCH#1#2{\setbox0=\hbox{\raise#1\hbox{#2}}\smash{\box0}}

\def\@evenfoot{}\def\@oddfoot{}

\def\@evenhead{\hbox to\textwidth{\small\rm\thepage \hfill
{\it Yueyue FENG, Qi LIU, Yuxin WANG and Jinyu XIA}}} 

\def\@oddhead{\hbox to \textwidth{\small{\it
The skew generalized Von Neumann-Jordan type  constant in Banach spaces
} \hfill\thepage}}   

\catcode`@=12

\def\bc{\begin{center}}
\def\ec{\end{center}}
\def\no{\noindent}
\def\hang{\hangindent\parindent}
\def\textindent#1{\indent\llap{\qquad #1\ \ \enspace}\ignorespaces}
\def\ref{\par\hang\textindent}

\begin{document}


\abovedisplayskip=6pt plus 1pt minus 1pt \belowdisplayskip=6pt
plus 1pt minus 1pt
\thispagestyle{empty} \vspace*{-1.0truecm} \noindent

\vskip 10mm \bc{\Large\bf The skew generalized Von Neumann-Jordan type  constant in Banach spaces    
\footnotetext{\footnotesize
Supported by Anhui Province Higher Education Science Research Project(Natural Science), 2023AH050487.\\
* Corresponding author\\  
 E-mail address:
} } \ec  

\vskip 5mm
\bc{\bf Yuxin WANG$^1$,\ \ \ Qi LIU$^{1,*}$,\ \ \ Yueyue FENG$^1$,\ \ \ Jinyu XIA$^1$,\ \ \ Muhammad SARFRAZ$^2$}\\  
{\small\it $1$. School of Mathematics and physics, Anqing Normal University,  Anqing $246133$, P. R. China\\$2$.School of Mathematics and System Science,	Xinjiang University, Urumqi 830046, P.R.China
}\ec   
\vskip 1 mm

{\narrower\noindent{\small {\small\bf Abstract}\ \ 
Recently, the Von Neumann-Jordan type constants $C_{-\infty}(X)$ has defined by Takahashi. A new skew generalized constant $C^p_{-\infty}(\lambda,\mu,X)$ based on  $C_{-\infty}(X)$ constant is given in this paper,. First, we will obtain some basic properties of this new constant. Moreover, some relations between this new constant and other constants are investigated. Specially, with the Banach-Mazur distance, we use this new constant to study isomorphic Banach spaces. Ultimately, by leveraging the connection between the newly introduced constant and the weak orthogonality coefficient $\omega(X)$, a sufficient condition for normal structure is established.

\vspace{1mm}\baselineskip 12pt

\no{\small\bf Keywords} \ \ Banach spaces; geometric constants;uniformly non-square; Banach-Mazur distance; normal structure coefficient

\par

\vspace{2mm}

\no{\small\bf MR(2020) Subject Classification\ \ {\rm 46B20}} 

}}

\baselineskip 15pt

\section{1. Introduction}

In recent times, there has been a surge in research on various geometric constants for Banach spaces, with particular emphasis on the constant \( C_{\mathrm{NJ}}(X) \) named von Neumann-Jordan constant and the \( J(X) \) constant named James constant. Notable contributions to this field have been made by Gao \cite{01,02}, Yang and Wang \cite{03}, and Kato, Maligranda, and Takahashi \cite{04,05}. For more detailed information, readers are encouraged to consult \cite{06,07,08,09} and the references contained therein.

The paper is organized as follows:   

In the next section we just recall a few basic definitions and related properties.  

In Section 3 some equivalent forms of $C^p_{-\infty}(\lambda,\mu,X)$ are considered. Specially, the new constant can be employed to establish a connection between a Banach space \( X \) and its dual space \( X^* \). Then we obtain some relations between this new constant and other well-known constants, including the  $C^p_{NJ}(\lambda,\mu,X)$, $C_{NJ}(X)$ and $J(X)$ constants. Specially, we will give an equivalent relationship among the constants of $J(X),~C_{NJ}(X),~C_{-\infty}^p(\lambda, \mu, X)$ and an inequality connection linking the $C_{-\infty}^p(\lambda, \mu, X)$ constant to the $J(X)$ constant. Based on this inequality, as two corollaries we also obtain the relations with the $C_{-\infty}^p(\lambda, \mu, X)$ constant and uniformly non-square.

In Section 4, we utilize the Banach–Mazur distance $d(X,Y)$ to compare the constant $C^p_{-\infty}(\lambda,\mu,X)$ between space $X$ and its isomorphic counterpart $Y$.

Finally, in the last Section 5, we  establish a new sufficient condition for normal structure of Banach spaces in terms of $C^p_{-\infty}(\lambda,\mu,X)$.

\section{2. Notations and Preliminaries}
Throughout the paper, we consider a real Banach space $X$ with $\operatorname{dim}X \geq 2$. The unit ball of $X$ is denoted by $B_X$, and its unit sphere by $S_X$.

We recall the definitions of some geometric constants.
\begin{definition}
For a Banach space $X$, if there exists a $\delta \in (0, 1)$ such that for any $x_1, x_2 \in S_X$,   at least one of the following inequalities holds:  
\[
\frac{\|x_1 + x_2\|}{2} \leq 1 - \delta \quad \text{or} \quad \frac{\|x_1 - x_2\|}{2} \leq 1 - \delta,
\]   
then $X$ is said to be uniformly non-square. 
	
	Related to this property is the James (or non-square ) constant of $X$, denoted by  
	\[
	J(X) = \sup \{\min\{\|x_1 + x_2\|, \|x_1 - x_2\|\} : x_1, x_2 \in S_X\}.
	\]  
	
	Importantly, a Banach space $X$ is uniformly non-square if and only if the James constant satisfies $J(X) < 2$.
\end{definition}
In \cite{16}, Clarkson presented the Von Neumann-Jordan constant as the smallest constant $C$ for which $$\frac{1}{C}\leq\frac{\|x_1+x_2\|^2+\|x_1-x_2\|^2}{2\|x_1\|^2+2\|x_2\|^2}\leq C,$$where $(x_1,x_2)\neq(0,0)$.

Furthermore, regarding the Von Neumann-Jordan constant, an equivalent definition is as follows:$$C_{NJ}(X)=\sup\left\{\frac{\|x_1+x_2\|^2+\|x_1-x_2\|^2}{2\|x_1\|^2+2\|x_2\|^2}:x_1,x_2\in X, (x_1,x_2)\neq(0,0)\right\}.$$

Later, through the study of the Von Neumann-Jordan constant, Takahashi defined the von Neumann-Jordan type constants as follows:$$C_{-\infty}(X)=\sup\left\{\frac{\min\{\|x_1+x_2\|^2,\|x_1-x_2\|^2\}}{2\|x_1\|^2+2\|x_2\|^2}:x_1,x_2\in X, (x_1,x_2)\neq(0,0)\right\}.$$

Subsequently, scholars generalized the above two kinds of constants into the following forms:

In \cite{17}, Liu et al. introduced the generalized Von Neumann-Jordan constant: For $\lambda,~\mu>0$,$$L_{YJ}(\lambda,\mu,X)=\sup\left\{\frac{\|\lambda x_1+\mu x_2\|^{2}+\|\mu x_1-\lambda x_2\|^{2}}{(\lambda^2+\mu^2)(\|x_1\|^{2}+\|x_2\|^{2})}:x_1,x_2\in X,(x_1,x_2)\neq(0,0)\right\}.$$

In \cite{18}, Ni et al. introduced the skew generalized Von Neumann-Jordan constant: For $\lambda,~\mu>0$ and $1\leq p<+\infty$, $$C_{NJ}^{p}(\lambda,\mu,X)=\sup\left\{\frac{\|\lambda x_1+\mu x_2\|^{p}+\|\mu x_1-\lambda x_2\|^{p}}{2^{p-2}(\lambda^{p}+\mu^{p})(\|x_1\|^{p}+\|x_2\|^{p})}:x_1,x_2\in X,(x_1,x_2)\neq(0,0)\right\}.$$

In \cite{15}, Zuo introduced the $C^p_{-\infty}(X)$ constant:$$C^p_{-\infty}(X)=\sup\left\{\frac{\min\{\|x_1+x_2\|^p,\|x_1-x_2\|^p\}}{2^{p-2}\|x_1\|^p+2^{p-2}\|x_2\|^2}:x_1,x_2\in X, (x_1,x_2)\neq(0,0)\right\}.$$

\section{3. The $C_{-\infty}^p(\lambda, \mu, X)$ constant}
Based on the above $C_{-\infty}(X)$ and $C^p_{-\infty}(X)$ constant, we define the following skew generalized $C_{-\infty}^p(\lambda, \mu, X)$  constant:
	\begin{definition}
			Let $X$ be a Banach space, for $\lambda>0,\mu>0$ and $1\leq p<+\infty$, the constant $C_{-\infty}^p(\lambda, \mu, X)$ is defined as
		$$C_{-\infty}^p(\lambda, \mu, X) = \sup\left\{\frac{\min\{\| \lambda x_1 + \mu x_2 \|^p , \| \mu x_1 - \lambda x_2 \|^p\}}{2^{p-3} (\lambda^p + \mu^p) (\| x_1 \|^p + \| x_2 \|^p)} : \, x_1, x_2 \in X,(x_1,x_2)\neq(0,0)\right\}.$$
	\end{definition}
Clearly the $C^p_{-\infty}(\lambda,\mu,X)$ constant also can be rewritten as the following form: For $\lambda>0,\mu>0$ and $1\leq p<+\infty$,
$$C^p_{-\infty}{(\lambda,\mu, X)}=\sup \bigg\{\frac{\min\{\Vert \lambda x_1+\mu x_2\Vert^p,\Vert \mu x_1-\lambda x_2\Vert^p\}}{2^{p-3} (\lambda^p + \mu^p) (\| x_1 \|^p + \| x_2 \|^p)}:  x_1,x_2\in X, \Vert x_1\Vert=1, \Vert x_2\Vert\leq 1 \bigg\}.$$
or equivalently, for $\lambda>0,\mu>0$ and $1\leq p<+\infty$,
$$C^p_{-\infty}{(\lambda,\mu, X)}=\sup \bigg\{\frac{\min\{\Vert \lambda x+\mu ty\Vert^p,\Vert \mu x-\lambda ty\Vert^p\}}{2^{p-3} (\lambda^p + \mu^p) (1+ t^p)}: x_1,x_2\in S_X,
0\leq t\leq 1\bigg\}.$$

Moreover, the following proposition establishes another form of $C^p_{-\infty}(\lambda,\mu, X)$:

\begin{Proposition}
	Let $X$ be a Banach space,  for $\lambda>0,\mu>0$ and $1\leq p<+\infty$, $$C_{-\infty}^p(\lambda, \mu, X) = \sup\left\{\frac{\min\{\| \lambda x_1 + \mu x_2 \|^p , \| \mu x_1 - \lambda x_2 \|^p\}}{2^{p-2} (\lambda^p + \mu^p)} :\|x_1\|^p+\|x_2\|^p=2  \right\}.$$
\end{Proposition}
\begin{proof}
	Let $$K=\sup\left\{\frac{\min\{\| \lambda x_1 + \mu x_2 \|^p , \| \mu x_1 - \lambda x_2 \|^p\}}{2^{p-2} (\lambda^p + \mu^p)} :\|x_1\|^p+\|x_2\|^p=2\right\}.$$ Then clearly, $C_{-\infty}^p(\lambda, \mu, X) \geq K$.
	
	Now we prove that $C_{-\infty}^p(\lambda, \mu, X) \leq K$, let $x_1,x_2 \in S_{X}$, and let $0 \leq t \leq 1$, put 
	$$\alpha=\frac{2^{\frac{1}{p}}x_1}{(1+t^p)^{\frac{1}{p}}},\beta=\frac{2^{\frac{1}{p}}tx_2}{(1+t^p)^{\frac{1}{p}}}.$$\\
	Then $\|\alpha\|^p+\|\beta\|^p=2$, and we obtain that\\
	$$\frac{\min\{\| \lambda x_1 + \mu tx_2 \|^p , \| \mu x_1 - \lambda tx_2 \|^p\}}{2^{p-3} (\lambda^p + \mu^p) (1+t^p)}=\frac{\min\{\| \lambda\alpha + \mu\beta \|^p , \| \mu\alpha - \lambda\beta \|^p\}}{2^{p-2} (\lambda^p + \mu^p)} \leq K,$$\\
	which implies that $C_{-\infty}^p(\lambda, \mu, X) \leq K$.
	
	Hence,	$$C_{-\infty}^p(\lambda, \mu, X) = \sup\left\{\frac{\min\{\| \lambda x_1 + \mu x_2 \|^p , \| \mu x_1 - \lambda x_2 \|^p\}}{2^{p-2} (\lambda^p + \mu^p)} :\|x_1\|^p+\|x_2\|^p=2  \right\}.$$
\end{proof}
	Next, we are going to present the upper and lower bounds of this constant. Before we start, we need the following simple lemma.
	\begin{Lemma}\label{l1}
			If $\lambda>0,\mu>0$, then $\lambda^p+\mu^p \leq (\lambda+\mu)^p \leq 2^{p-1}(\lambda^p+\mu^p)$ is always valid for any $1\leq p<+\infty$.
	\end{Lemma}

	\begin{Proposition}\label{p3}
	Suppose that $X$ is a nontrivial Banach space. Then   $$ \frac{\min\{\mu^p, \lambda^p\}}{2^{p-3} (\lambda^p + \mu^p)} \leq C_{-\infty}^p(\lambda, \mu, X) \leq 2.$$
	\end{Proposition}
	
	\begin{proof}
	First, by Lemma \ref{l1}, we have
	$$\begin{aligned}	(\lambda\|x_1\|+\mu\|x_2\|)^p+(\mu\|x_1\|+\lambda\|x_2\|)^p &\leq 2^{p-1}(\lambda^p\|x_1\|^p+\mu^p\|x_2\|^p+\mu^p\|x_1\|^p+\lambda^p\|x_2\|^p)	\\&= 2^{p-1} (\lambda^p + \mu^p) (\|x_1\|^p+\|x_2\|^p).
	\end{aligned}$$
	Since
	$$\min\{\|\lambda x_1+\mu x_2\|^p,\|\mu x_1-\lambda x_2\|^p\} \leq \frac{\|\lambda x_1+\mu x_2\|^p+\|\mu x_1-\lambda x_2\|^p}{2},$$
	we have$$\begin{aligned}
		\frac{\min\{\| \lambda x_1 + \mu x_2 \|^p , \| \mu x_1- \lambda x_2 \|^p\}}{2^{p-3} (\lambda^p + \mu^p) (\| x_1\|^p + \| x_2 \|^p)}&\leq \frac{\| \lambda x_1 + \mu x_2 \|^p + \| \mu x_1- \lambda x_2 \|^p}{2^{p-2} (\lambda^p + \mu^p) (\| x_1\|^p + \| x_2 \|^p)}\\&\leq \frac{(\lambda\|x_1\|+\mu\|x_2\|)^p+(\mu\|x_1\|+\lambda\|x_2\|)^p}{2^{p-2} (\lambda^p + \mu^p) (\| x_1\|^p + \| x_2 \|^p)}\\&\leq 
		\frac{2^{p-1}(\lambda^p+\mu^p)(\|x_1\|^p+\|x_2\|^p)}{2^{p-2}(\lambda^p+\mu^p)(\|x_1\|^p+\|x_2\|^p)}
		\\&\leq
		2.
	\end{aligned}$$
Hence we obtain $ C_{-\infty}^p(\lambda, \mu, X) \leq 2$.
	
	On the other hand, let $x=0,y\neq 0$, we have
	$$\begin{aligned}
		\frac{\min\{\| \lambda x_1 + \mu x_2 \|^p , \| \mu x_1- \lambda x_2 \|^p\}}{2^{p-3} (\lambda^p + \mu^p) (\| x_1\|^p + \| x_2 \|^p)}=\frac{\min\{\mu^p\|x_2\|^p,\lambda^p\|x_2\|^p\}}{2^{p-3}(\lambda^p+\mu^p)\|x_2\|^p}		=\frac{\min\{\mu^p,\lambda^p\}}{2^{p-3}(\lambda^p+\mu^p)}.
	\end{aligned}$$
	Therefore, $ C_{-\infty}^p(\lambda, \mu, X)\geq\frac{\min\{\mu^p, \lambda^p\}}{2^{p-3} (\lambda^p + \mu^p)}$ holds.
	\end{proof}
	\begin{Corollary}
	If $\lambda=\mu=1$ and $p=1$, then $C_{-\infty}^1(1,1, X)=2$.
	\end{Corollary}
	\begin{proof}
	If $\lambda=\mu=1$ and $p=1$, then $\frac{\min\{\mu^p, \lambda^p\}}{2^{p-3} (\lambda^p + \mu^p)}=2$ holds.
Then according to Proposition \ref{p3}, $C_{-\infty}^1(1,1, X)=2$ is valid.\end{proof}

	Next we will give the estimate of the $C_{-\infty}^p(\lambda, \mu, X)$ constant in some
	specific Banach spaces when $p > 1$.

		\begin{Example}Consider $X = (\mathbb{R}^2, \|\cdot\|_1)$, that is $X=R^2$ endowed with the norm $\|(a, b)\|= |a|+|b|$. Then ${C^{P}_{-\infty}}(1,1,X)=2 $.
			
	 Let $a=(1,0),b=(0,1)$, then we have $\|a\|=\|b\|=1,\|a+b\|=2$, and $\|a-b\|=2$, hence
	 $$\frac{\min\{\|a+b\|^p,\|a-b\|^p\}}{2^{p-2}(\|a\|^p+\|b\|^p)}=2,$$
	 which implies that ${C^{P}_{-\infty}}(1,1,X)\geq2 $. Thus ${C^{P}_{-\infty}}(1,1,X)=2$ according to Proposition \ref{p3}.\end{Example}
	
		\begin{Example} Consider  $X=(\mathbb{R}^2, \|\cdot\|_\infty)$, that is $X=R^2$ endowed with the norm $\|a, b\| = \max\{|a|, |b|\}$. Then
        ${C^{P}_{-\infty}}(1,1,X)=2 $. 
	   	
	   	 Let $a=(1,1),b=(1,-1)$, we have $\|a\|=\|b\|=1,\|a+b\|=2$, and $\|a-b\|=2$, hence
	   	$$\frac{\min\{\|a+b\|^p,\|a-b\|^p\}}{2^{p-2}(\|a\|^p+\|b\|^p)}=2,$$\\
	   which implies that ${C^{P}_{-\infty}}(1,1,X)\geq2 $. Thus ${C^{P}_{-\infty}}(1,1,X)=2$ according to Proposition \ref{p3}.\end{Example}

	\begin{Theorem}
	Let $X$ be a Banach space. Then 
		 $$C_{-\infty}^p(\lambda, \mu, X) \leq C_{NJ}^p(\lambda, \mu, X)\leq\frac{1}{2}C_{-\infty}^p(\lambda, \mu, X)+\frac{4\max\{\lambda^p,\mu^p\}}{\lambda^p+\mu^p}.$$
\end{Theorem}

	 \begin{proof}
	On the one hand, it is immediately demonstrable by
		$$
		\frac{\min\{\| \lambda x_1 + \mu x_2 \|^p , \| \mu x_1- \lambda x_2 \|^p\}}{2^{p-3} (\lambda^p + \mu^p) (\| x_1\|^p + \| x_2 \|^p)}
		\leq \frac{\| \lambda x_1 + \mu x_2 \|^p + \| \mu x_1- \lambda x_2 \|^p}{2^{p-2} (\lambda^p + \mu^p) (\| x_1\|^p + \| x_2 \|^p)}.
		$$
	One the other hand, obviously, the $C_{-\infty}^p(\lambda, \mu, X)$ constant can be denoted as the following form:	$$C_{-\infty}^p(\lambda, \mu, X) = \sup\left\{\frac{\min\{\| \lambda x_1 + \mu tx_2 \|^p , \| \mu x_1 - \lambda tx_2 \|^p\}}{2^{p-3} (\lambda^p + \mu^p) (1+t^p)} : \, x_1, x_2 \in S_X,0\leq t\leq 1 \right\},$$where $\lambda>0,\mu>0$. Then since $$\|\lambda x_1+\mu tx_2\|^p+\|\mu x_1-\lambda tx_2\|^p\leq\min\{\|\lambda x_1+\mu tx_2\|^p,\|\mu x_1-\lambda tx_2\|^p\}+\max\{\lambda^p,\mu^p\}(1+t)^p$$ for all $x_1,x_2\in S_X$, then by Lemma \ref{l1}, we obtain that$$\begin{aligned}C_{NJ}^{p}(\lambda,\mu,X)&\leq\frac{1}{2}C^p_{-\infty}(\lambda,\mu,X)+\frac{\max\{\lambda^p,\mu^{p}\}(1+t)^{p}}{2^{p-3}(\lambda^{p}+\mu^{p})(1+t^{p})}\\&\leq\frac{1}{2}C^p_{-\infty}(\lambda,\mu,X)+\frac{2^{p-1}\max\{\lambda^{p},\mu^{p}\}(1+t^{p})}{2^{p-3}(\lambda^{p}+\mu^{p})(1+t^{p})}\\&=\frac{1}{2}C^p_{-\infty}(\lambda,\mu,X)+\frac{4\max\{\lambda^{p},\mu^{p}\}}{\lambda^{p}+\mu^{p}},\end{aligned}$$
	which implies that $$ C_{NJ}^p(\lambda, \mu, X)\leq\frac{1}{2}C_{-\infty}^p(\lambda, \mu, X)+\frac{4\max\{\lambda^p,\mu^p\}}{\lambda^p+\mu^p}.$$
	This completes the proof.

		\end{proof}
\begin{Corollary}
Assume that the dimension of $X$ is finite and $C_{-\infty}^p(\lambda, \mu, X) = C_{NJ}^p(\lambda, \mu, X)$. Then, there exist vectors $x_1, x_2 \in X$ such that $\| \lambda x_1 + \mu x_2\| = \| \mu x_1 - \lambda x_2 \|$ and both constants achieve their respective suprema.
\end{Corollary}	
		\begin{proof}
				For any $x_1,x_2 \in X$,we have 
			$$
			\frac{\min\{\| \lambda x_1 + \mu x_2 \|^p , \| \mu x_1- \lambda x_2 \|^p\}}{2^{p-3} (\lambda^p + \mu^p) (\| x_1\|^p + \| x_2 \|^p)}
			\leq \frac{\| \lambda x_1 + \mu x_2 \|^p + \| \mu x_1- \lambda x_2 \|^p}{2^{p-2} (\lambda^p + \mu^p) (\| x_1\|^p + \| x_2 \|^p)}
			$$
			and the identity holds if and only if $\| \lambda x_1 + \mu x_2\| = \| \mu x_1 - \lambda x_2 \|$.
			
			Now, we suppose that dim $X < +\infty$ and let $x_1,x_2 \in X$ be such that
			$$C_{-\infty}^p(\lambda, \mu, X)=\frac{\min\{\| \lambda x_1 + \mu x_2 \|^p , \| \mu x_1- \lambda x_2 \|^p\}}{2^{p-3} (\lambda^p + \mu^p) (\| x_1\|^p + \| x_2 \|^p)}.$$
			
			Then if $C_{-\infty}^p(\lambda, \mu, X) = C_{NJ}^p(\lambda, \mu, X)$, we obtain that $$C_{-\infty}^p(\lambda, \mu, X) = 	\frac{\min\{\| \lambda x_1 + \mu x_2 \|^p , \| \mu x_1- \lambda x_2 \|^p\}}{2^{p-3} (\lambda^p + \mu^p) (\| x_1\|^p + \| x_2 \|^p)} \leq \frac{\| \lambda x_1 + \mu x_2 \|^p + \| \mu x_1- \lambda x_2 \|^p}{2^{p-2} (\lambda^p + \mu^p) (\| x_1\|^p + \| x_2 \|^p)} \leq     C_{NJ}^p(\lambda, \mu, X).$$
			
			Therefore, it is proven that.
		\end{proof}
	Next, based on Hahn-Banach theorem, we will establish an inequality relationship between $C_{-\infty}^p(\lambda, \mu, X)$ and $C_{-\infty}^p(\lambda, \mu, X^*)$.
	\begin{Theorem}
		Let $X$ be a Banach space. Then$$ \frac{1}{2^{p-2}}C_{-\infty}^p(\lambda, \mu, X)-\frac{(\lambda+\mu)^p}{2^{p-2}(\lambda^p+\mu^p)}\leq C_{-\infty}^p(\lambda, \mu, X^*)\leq2^{p-2}C_{-\infty}^p(\lambda, \mu, X)+\frac{(\lambda+\mu)^p}{\lambda^p+\mu^p}.$$
	\end{Theorem}
	\begin{proof}
	  First, for any $\varepsilon>0$, there exist $x_1,x_2\in S_{X}$ such that
		$$\min\{\|\lambda x_1+\mu x_2\|^p,\|\mu x_1-\lambda x_2\|^p\}\geq 2^{p-2}(\lambda^p+\mu^p)C_{-\infty}^p(\lambda, \mu, X)-\varepsilon.$$
		In addition, according to Hahn-Banach theorem, there exist $f,g\in S_{X^*}$ such that
		$$f(\lambda x+\mu y)=\|\lambda x_1+\mu x_2\|,\quad g(\mu x-\lambda y)=\|\mu x_1-\lambda x_2\|.$$
		Then, we have$$
		\begin{aligned}C^p_{-\infty}(\lambda,\mu,X^*)&\geq \frac{\min\{\|\lambda f+\mu g\|^p,\| \mu f-\lambda g\|^p\}}{2^{p-2}\left(\lambda^p+\mu^p\right)}
			\\& \geq \frac{\|\lambda f+\mu g\|^p+\| \mu f-\lambda g\|^p}{2^{p-2}\left(\lambda^p+\mu^p\right)}-\frac{\max\{\|\lambda f+\mu g\|^p,\| \mu f-\lambda g\|^p\}}{2^{p-2} (\lambda^p+\mu^p)} \\
			& \geq\frac{\|\lambda f+\mu g\|^p+\| \mu f-\lambda g\|^p}{2^{p-2}\left(\lambda^p+\mu^p\right)}-\frac{(\lambda+\mu)^p}{2^{p-2}\left(\lambda^p+\mu^p\right)} \\
			&\geq\frac{[(\lambda f+\mu g)(x)+(\mu f-\lambda g)(y)]^p}{2^{2 p-3}\left(\lambda^p+\mu^p\right)}
				-\frac{(\lambda+\mu)^p}{2^{p-2}\left(\lambda^p+\mu^p\right)}\\&=\frac{[\|\lambda x_1+\mu x_2\|+\|\mu x_1-\lambda x_2\|]^p}{2^{2 p-3}\left(\lambda^p+\mu^p\right)}
				-\frac{(\lambda+\mu)^p}{2^{p-2}\left(\lambda^p+\mu^p\right)}\\&\geq\frac{\min\{\|\lambda x_1+\mu x_2\|^p,\|\mu x_1-\lambda x_2\|^p\}}{2^{2 p-4}\left(\lambda^p+\mu^p\right)}
				-\frac{(\lambda+\mu)^p}{2^{p-2}\left(\lambda^p+\mu^p\right)}\\&\geq\frac{2^{p-2}(\lambda^p+\mu^p)C_{-\infty}(\lambda,\mu,X)-\varepsilon}{2^{2 p-4}\left(\lambda^p+\mu^p\right)}
				-\frac{(\lambda+\mu)^p}{2^{p-2}\left(\lambda^p+\mu^p\right)}.
		\end{aligned}
		$$
		Then let $\varepsilon\to 0$, we have $$ C_{-\infty}^p(\lambda, \mu, X^*)\geq\frac{1}{2^{p-2}}C_{-\infty}^p(\lambda, \mu, X)-\frac{(\lambda+\mu)^p}{2^{p-2}(\lambda^p+\mu^p)}.$$
	Conversely, given $f, g \in S_{X^*}$, for any $\varepsilon > 0$, we can find $x_1, x_2 \in S_X$ such that
		$$
		(\lambda f+\mu g)(x_1)>\|\lambda f+\mu g\|-\varepsilon, \quad(\mu f-\lambda g)(x_2)>\|\mu f-\lambda g\|-\varepsilon
		$$
		Thus, for any $f, g \in S_{X^*}$, we have
		$$
		\begin{aligned}
		\frac{\min \{\|\lambda f+\mu g\|^p,\|\mu f-\lambda g\|^p\}}{2^{p-2}(\lambda^p+\mu^p)} & \leq 	\frac{ \|\lambda f+\mu g\|^p+\|\mu f-\lambda g\|^p}{2^{p-1}(\lambda^p+\mu^p)}  \\
			& \leq\frac{ [\|\lambda f+\mu g\|+\|\mu f-\lambda g\|]^p}{2^{p-1}(\lambda^p+\mu^p)}
			 \\& <\frac{ [(\lambda f+\mu g)(x)+(\mu f-\lambda g)(y)+2\varepsilon]^p}{2^{p-1}(\lambda^p+\mu^p)} 
			 \\& =\frac{[f(\lambda x-\mu y)+g(\mu x-\lambda y)+2\varepsilon]^p}{2^{p-1}(\lambda^p+\mu^p)} \\&\leq\frac{[\|f\|\|\lambda x_1+\mu x_2\|+\|g\|\|\mu x_1-\lambda x_2\|+2\varepsilon]^p}{2^{p-1}(\lambda^p+\mu^p)}\\&\leq\frac{[\|\lambda x_1+\mu x_2\|+\|\mu x_1-\lambda x_2\|+2\varepsilon]^p}{2^{p-1}(\lambda^p+\mu^p)},
		\end{aligned}
		$$
		let $\varepsilon\to 0$, we have $$\begin{aligned}
			\frac{[\|\lambda x_1+\mu x_2\|+\|\mu x_1-\lambda x_2\|]^p}{2^{p-1}(\lambda^p+\mu^p)}&\leq\frac{\|\lambda x_1+\mu x_2\|^p+\|\mu x_1-\lambda x_2\|^p}{\lambda^p+\mu^p}\\&\leq\frac{\min\{\|\lambda x_1+\mu x_2\|^p,\|\mu x_1-\lambda x_2\|^p\}+(\lambda+\mu)^p}{\lambda^p+\mu^p}\\&\leq2^{p-2}C_{-\infty}^p(\lambda, \mu, X)+\frac{(\lambda+\mu)^p}{\lambda^p+\mu^p}.
		\end{aligned}$$
		This implies that $$C_{-\infty}^p(\lambda, \mu, X^*)\leq2^{p-2}C_{-\infty}^p(\lambda, \mu, X)+\frac{(\lambda+\mu)^p}{\lambda^p+\mu^p}.$$
	\end{proof}
	\begin{Theorem}
		Let $X$ be a Banach space.Then the following conclusions are equivalent:
		
		$(i)J(X)=2;$
	
		$(ii)C_{NJ}(X)=2;$
		
		$(iii)C_{-\infty}^{p}(\lambda,\mu,X)=\frac{(\lambda+\mu)^p}{2^{p-2}(\lambda^p+\mu^p)}.$
	\end{Theorem}
	\begin{proof}
		Since it has been proven that $J(X)=2$ and $C_{NJ}(X)=2$ are equivalent, see \cite{20}. We now prove that $C_{NJ}(X)=2$ and $C_{-\infty}^{p}(\lambda,\mu,X)=\frac{(\lambda+\mu)^p}{2^{p-2}(\lambda^p+\mu^p)}$
		are also equivalent.\\
		Assume that $C_{NJ}(X)=2$, which means that there exist sequence of
		points ${{x_1}_{n}} \in S_{X}$, ${{x_2}_{n}} \in B_{X}$ such that\\ $$\frac{\|{{x_1}_{n}} +{{x_2}_{n}} \|^2+\|{{x_1}_{n}} -{{x_2}_{n}} \|^2}{2(1+\|{{x_2}_{n}} \|^2)} \rightarrow 2 (n \rightarrow \infty),$$
		at this time, only\\
		$\|{{x_1}_{n}} +{{x_2}_{n}} \| \rightarrow 2 (n \rightarrow \infty)$ , $\|{{x_1}_{n}} -{{x_2}_{n}} \| \rightarrow 2 (n \rightarrow \infty)$ and $\|{{x_2}_{n}} \| \rightarrow  1 (n \rightarrow \infty)$, \\the above  is satisfied.
	    Next we prove that
	    	$$\frac{\min\{\| \lambda {{x_1}_{n}}  + \mu {{x_2}_{n}}  \|^p , \| \mu {{x_1}_{n}}  - \lambda {{x_2}_{n}}  \|^p\}}{2^{p-3} (\lambda^p + \mu^p) (1 + \| {{x_2}_{n}}  \|^p)} \rightarrow \frac{(\lambda+\mu)^p}{2^{p-2}(\lambda^p+\mu^p)}
	        (n \rightarrow \infty).$$
	   Obviously, the $C_{-\infty}^{p}(\lambda,\mu,X)$ constant can ba presented as the following form:
	   $$C_{-\infty}^p(\lambda, \mu, X) = \sup\left\{\frac{\min\{\| \lambda x_1 + \mu x_2 \|^p , \| \mu x_1- \lambda x_2 \|^p\}}{2^{p-3} (\lambda^p + \mu^p) (1 + \| x_2\|^p)} : x_1 \in S_{X},x_2 \in B_{X} \right\}.$$
	   We consider the following two cases.\\
	  \textbf{Case i}: if $\lambda=\mu$, then
	   $$\frac{\min\{\| \lambda {{x_1}_{n}}  + \mu {{x_2}_{n}}  \|^p , \| \mu {{x_1}_{n}}  - \lambda {{x_2}_{n}}  \|^p\}}{2^{p-3} (\lambda^p + \mu^p) (1 + \| {{x_2}_{n}}  \|^p)} = \frac{\min\{\| {{x_1}_{n}}  +  {{x_2}_{n}}  \|^p , \| {{x_1}_{n}}  - {{x_2}_{n}}  \|^p\}}{2^{p-2}  (1 + \| {{x_2}_{n}}  \|^p)} \rightarrow 2 (n \rightarrow \infty).$$ 
	   Hence it is valid.\\
	   \textbf{Case ii}: if $\lambda\neq\mu$, that is $\lambda < \mu$ or $\lambda > \mu$. We now consider the situation where
	   $\lambda > \mu$ without loss of generality (the situation where $\lambda < \mu$ is similar).\\
	   Since
	   $$\|\lambda {{x_1}_{n}} +\mu {{x_2}_{n}} \|\leq \lambda\|{{x_1}_{n}} \|+\mu\|{{x_2}_{n}} \|,$$
	  and $$\|\lambda {{x_1}_{n}} +\mu {{x_2}_{n}} \|=\|\lambda({{x_1}_{n}} -{{x_2}_{n}} )-(\lambda-\mu){{x_2}_{n}} \|\geq\lambda\|{{x_1}_{n}} -{{x_2}_{n}} \|-(\lambda-\mu)\|{{x_2}_{n}} \|.$$
	  Furthermore, we can deduce that $$\lambda\|{{x_1}_{n}} \|+\mu\|{{x_2}_{n}} \|\to \lambda+\mu (n\to\infty)~~\text{and}~~ \lambda\|{{x_1}_{n}} -{{x_2}_{n}} \|-(\lambda-\mu)\|{{x_2}_{n}} \|\to \lambda+\mu (n\to\infty),$$ which implies that $\|\lambda {{x_1}_{n}} +\mu {{x_2}_{n}} \|\to \lambda+\mu (n\to\infty).$
	  \\Then since  $$\|\mu {{x_1}_{n}} -\lambda {{x_2}_{n}} \|\leq \lambda\|{{x_1}_{n}} \|+\mu\|{{x_2}_{n}} \|,$$ and $$\|\mu {{x_1}_{n}} -\lambda {{x_2}_{n}} \|=\|(\lambda+\mu)({{x_1}_{n}} -{{x_2}_{n}} )+\mu {{x_2}_{n}} -\lambda {{x_1}_{n}} \|\geq(\lambda+\mu)\|{{x_1}_{n}} -{{x_2}_{n}} \|-\mu\|{{x_2}_{n}} \|-\lambda\|{{x_1}_{n}} \|.$$Furthermore, we can deduce that $$\lambda\|{{x_1}_{n}} \|+\mu\|{{x_2}_{n}} \|\to \lambda+\mu (n\to\infty)~~\text{and}~~ (\lambda+\mu)\|{{x_1}_{n}} -{{x_2}_{n}} \|-\mu\|{{x_2}_{n}} \|-\lambda\|{{x_1}_{n}} \|\to \lambda+\mu (n\to\infty),$$ which implies that $\|\mu {{x_1}_{n}} -\lambda {{x_2}_{n}} \|\to \lambda+\mu (n\to\infty).$
	  \\Thus we obtain that $$\min\{\| \lambda {{x_1}_{n}}  + \mu {{x_2}_{n}}  \|^p , \| \mu {{x_1}_{n}}  - \lambda {{x_2}_{n}}  \|^p\}\to(\lambda+\mu)^p(n\to\infty).$$
	   Hence,
	   	$$\frac{\min\{\| \lambda {{x_1}_{n}}  + \mu {{x_2}_{n}}  \|^p , \| \mu {{x_1}_{n}}  - \lambda {{x_2}_{n}}  \|^p\}}{2^{p-3} (\lambda^p + \mu^p) (1 + \| {{x_2}_{n}}  \|^p)} \rightarrow \frac{(\lambda+\mu)^p}{2^{p-2}(\lambda^p+\mu^p)}
	   (n \rightarrow \infty).$$
	   Therefore, from the above two cases, $C_{NJ}(X)=2$ if and
	    only if $C_{-\infty}^{p}(\lambda,\mu,X)=\frac{(\lambda+\mu)^p}{2^{p-2}(\lambda^p+\mu^p)}$.\\
	    In conclusion, (i), (ii) and (iii) are equivalent.    
	\end{proof}
	\begin{Theorem}\label{t3}
		Let $X$ be a Banach space. Then$$\frac{[\max\{\lambda,\mu\}]J(X)-|\lambda-\mu|]^{p}}{2^{p-2}(\lambda^{p}+\mu^{p})}\leq C_{-\infty}^{p}(\lambda,\mu,X)\leq\frac{[\min\{\lambda,\mu\}J(X)+|\lambda-\mu|]^{p}}{2^{p-2}(\lambda^{p}+\mu^{p})}.$$
	\end{Theorem}
	\begin{proof}
		On the one hand, for any $x_1,x_2\in S_X$, since	$$\begin{aligned}\min\{\|\lambda x_1+\mu x_2\|^p,\|\mu x_1-\lambda x_2\|^p\}\leq2^{p-2}(\lambda^{p}+\mu^{p})C_{-\infty}^{(p)}(\lambda,\mu,X),\end{aligned}$$
		this is equivalent to $$\min\{\|\lambda x_1+\mu x_2\|,\|\mu x_1-\lambda x_2\|\}\leq\sqrt[p]{2^{p-2}\left(\lambda^{p}+\mu^{p}\right)C_{-\infty}^{p}(\lambda,\mu,X)}.$$
		Then by
		$$\|\lambda x_1+\mu x_2\|\geq\max\{\lambda,\mu\}\|x_1+x_2\|-|\lambda-\mu|$$
		and
		$$\|\mu x_1-\lambda x_2\|\geq\max\left\{\lambda,\mu\right\}\|x_1-x_2\|-|\lambda-\mu|,$$
		we obtain that
		$$\begin{aligned}&\sqrt[p]{2^{p-2}\left(\lambda^{p}+\mu^{p}\right)C_{-\infty}^{p}(\lambda,\mu,X)}\\\geq&\min\{\max\{\lambda,\mu\}\|x_1+x_2\|-|\lambda-\mu|,\max\{\lambda,\mu\}\|x_1-x_2\|-|\lambda-\mu|\}\\=&\max\{\lambda,\mu\}J(X)-|\lambda-\mu|,\end{aligned}$$
		which means that $$C_{-\infty}^{p}(\lambda,\mu,X)\geq\frac{[\max\{\lambda,\mu\}J(X)-|\lambda-\mu|]^{p}}{2^{p-2}\left(\lambda^p+\mu^p\right)}.$$
		Conversly, since
		$$\|\lambda x_1+\mu x_2\|\leq\min\{\lambda,\mu\}\|x_1+x_2\|+|\lambda-\mu|,$$
		and
		$$\|\mu x_1-\lambda x_2\|\leq\min\{\lambda,\mu\}\|x_1-x_2\|+|\lambda-\mu|.$$
		We obtain that $$\begin{aligned}&\frac{\min\{\| \lambda x_1 + \mu x_2 \|^p , \| \mu x_1- \lambda x_2 \|^p\}}{2^{p-2}(\zeta^{p}+v^{p})}\\\leq&\frac{[\min\{\min\{\lambda,\mu\}\|x_1+x_2\|+|\lambda-\mu|,\min\{\lambda,\mu\}\|x_1-x_2\|+|\lambda-\mu\|]^{p}}{2^{p-1}(\lambda^{p}+\mu^{p})}\\=&\frac{[\min\{\lambda,\mu\}]J(X)+|\lambda-\mu|^{p}}{2^{p-2}(\lambda^{p}+\mu^{p})}.\end{aligned}$$
		This implies that $$C_{-\infty}^{p}(\lambda,\mu,X)\leq\frac{[\min\{\lambda,\mu\}J(X)+|\lambda-\mu|]^{p}}{2^{p-2}(\lambda^{p}+\mu^{p})}.$$
	\end{proof}
	\begin{Corollary}\label{c1}
		$X$ is uniformly non-square if and only if  $$C_{-\infty}^{p}(\lambda,\mu,X)<\frac{(\lambda+\mu)^p}{2^{p-2}(\lambda^p+\mu^p)}.$$
	\end{Corollary}
	\begin{proof}
		If $X$ is uniformly non-square, then from what we have shown before, we know that $J(X)<2$ holds. This means that $$\frac{[\min\{\lambda,\mu\}J(X)+|\lambda-\mu|]^{p}}{2^{p-2}(\lambda^{p}+\mu^{p})}< \frac{(\lambda+\mu)^p}{2^{p-2}(\lambda^p+\mu^p)}.$$
		Thus by Theorem \ref{t3}, we obtain that$$C_{-\infty}^{p}(\lambda,\mu,X)<\frac{(\lambda+\mu)^p}{2^{p-2}(\lambda^p+\mu^p)}.$$
		
		Conversely, by Theorem \ref{t3}, we have $$\frac{[\max\{\lambda,\mu\}]J(X)-|\lambda-\mu|]^{p}}{2^{p-2}(\lambda^{p}+\mu^{p})}\leq C_{-\infty}^{p}(\lambda,\mu,X)<\frac{(\lambda+\mu)^p}{2^{p-2}(\lambda^p+\mu^p)}.$$ For$$ \frac{[\max\{\lambda,\mu\}]J(X)-|\lambda-\mu|]^{p}}{2^{p-2}(\lambda^{p}+\mu^{p})}<\frac{(\lambda+\mu)^p}{2^{p-2}(\lambda^p+\mu^p)},$$  with a simple computation, we can deduce that $J(X)<2$. Since $J(X)<2$ and $X$ is uniformly non-square are coincide, therefore when $C_{-\infty}^{p}(\lambda,\mu,X)<\frac{(\lambda+\mu)^p}{2^{p-2}(\lambda^p+\mu^p)}$ holds, $X$ is uniformly non-square.
		
		This completes the proof.
	\end{proof}
\section{4. Banach–Mazur distance and stability}
For isomorphic Banach spaces $X$ and $Y$, the Banach-Mazur distance $d(X,Y)$ is the infimum of $\|T\|\|T^{-1}\|$ over all isomorphisms $T$ mapping $X$ onto $Y$.
\begin{Theorem}\label{t4}
If $X$ and $Y$ are isomorphic Banach spaces. Then$$\frac{C^p_{-\infty}(\lambda,\mu,X)}{d(X,Y)^p}\leq C^p_{-\infty}(\lambda,\mu,Y)\leq C^p_{-\infty}(\lambda,\mu,X)d(X,Y)^p.$$
\end{Theorem}
\begin{proof}
Given $x_1, x_2 \in S_X$, for any $\epsilon > 0$, there is an isomorphism $T: X \to Y$ satisfying $\|T\|\|T^{-1}\| \leq (1 + \epsilon)d(X,Y)$. Set
$$x_1'=\frac{Tx_1}{\|T\|},\:x_2'=\frac{Tx_2}{\|T\|}.$$
Then $x_1^\prime,x_2^{\prime}\in B_Y$, since
$$\|x_1'\|=\frac{\|Tx_1\|}{\|T\|}\leq\|x_1\|=1,\:\|x_2'\|=\frac{\|Tx_2\|}{\|T\|}\leq\|x_2\|=1,$$
hence we obtain
$$\begin{aligned}\frac{\min\{\|\lambda x_1+\mu tx_2\|^{p},\|\mu x_1-\lambda tx_2\|^{p}\}}{2^{p-3}(\lambda^p+\mu^p)(1+t^{p})}&=\frac{\|T\|^{p}\min\{\|T^{-1}(\lambda x_1^{\prime}+\mu tx_2^{\prime})\|^{p},\|T^{-1}(\mu x_1^{\prime}-\lambda tx_2^{\prime})\|^{p}\}}{2^{p-3}(\lambda^p+\mu^p)(1+t^{p})}\\&\leq\frac{(1+\epsilon)^{p}d(X,Y)^{p}\min\{\|\lambda x_1^{\prime}+\mu tx_2^{\prime}\|^{p},\|\mu x_1^{\prime}-\lambda tx_2^{\prime}\|^{p}\}}{2^{p-3}(\lambda^p+\mu^p)(1+t^{p})}\\&\leq(1+\epsilon)^{p}d(X,Y)^{p}C^p_{-\infty}(\lambda,\mu,Y).\end{aligned}$$
Since $x_1,x_2\in S_X$ and $\epsilon>0$ are arbitrary, it follows that $$C^p_{-\infty}(\lambda,\mu,X)\leq d(X,Y)^pC^p_{-\infty}(\lambda,\mu,Y).$$
The second inequality follows by simply interchanging $X$ and $Y.$ 
		\end{proof}
\begin{Corollary}\label{c2}
Consider a nontrivial Banach space $X$ and let $X_{1} = (X, \| \cdot \|_{1})$, where $\| \cdot \|_{1}$ is an equivalent norm on $X$ such that for $a, b > 0$ and $x \in X$, we have  
\[
a \| x \| \leq \| x \|_{1} \leq b \| x \|.
\]
	Then $$\frac {a^{p}}{b^{p}}C^p_{- \infty }(\lambda,\mu, X) \leq C^p_{- \infty }( \lambda,\mu,X_{1}) \leq \frac {b^{p}}{a^{p}}C^p_{- \infty }(\lambda,\mu, X) .$$
\end{Corollary}
\begin{proof}
	This follows from the above Theorem \ref{t4} and the fact that $d(X,X_1)<\frac{b}{a}$.
\end{proof}
\begin{Remark}
	Corollary \ref{c2} can  also be understand in the following way.
	Let $\Vert \cdot\Vert_1$ and $\Vert \cdot\Vert_2$ be two equivalent norms 
	defined on the vector space $X$ such that for every $x\in X$,
	$$\Vert x\Vert_2\leq \Vert x\Vert_1\leq b\Vert x\Vert_2.$$
	Then  $C^p_{-\infty}(\lambda,\mu,X_1)\leq b^pC^p_{-\infty}(\lambda,\mu,X_2)$.

	For $x_1,x_2\in X$ and $x,y$ nor both zero, we have
	$$\begin{aligned}&\min \bigg\{\frac{\Vert \lambda x+\mu y\Vert_1^p}{2^{p-3}(\lambda^p+\mu^p)(\Vert x\Vert_1^p+\Vert y\Vert_1^p) }, ~~\frac{\Vert \mu x-\lambda y\Vert_1^p}{2^{p-3}(\lambda^p+\mu^p)(\Vert x\Vert_1^p+\Vert y\Vert_1^p)}\bigg\}\\\leq &\min \bigg\{\frac{b^p\Vert\lambda x+\mu y\Vert^p_2}{2^{p-3}(\lambda^p+\mu^p)(\Vert x\Vert_2^p+\Vert y\Vert_2^p) }, ~~\frac{b^p\Vert \mu x-\lambda y\Vert_2^p}{2^{p-3}(\lambda^p+\mu^p)(\Vert x\Vert_2^p+\Vert y\Vert_2^p)}\bigg\}\\\leq& b^pC^p_{-\infty}(\lambda,\mu,X_2).\end{aligned}$$
	Taking the supremum of the left-hand side of this inequality and the claim follows.
\end{Remark}
\section{5. The coefficient of weak orthogonality }
In 1948, Brodskii and Milman \cite{11} introduced the following geometric concepts:

\begin{definition}A Banach space $X$ is defined to have normal structure if, for every non-singleton closed, bounded, and convex subset $K \subseteq X$, the Chebyshev radius $r(K)$ is strictly smaller than the diameter $\operatorname{diam}(K)$. Specifically, the diameter $\operatorname{diam}(K)$ is given by  
	\[
	\operatorname{diam}(K) = \sup \{\|x - y\| : x, y \in K\},
	\]  
	and the Chebyshev radius $r(K)$ is defined as  
	\[
	r(K) = \inf \{\sup \{\|x - y\| : y \in K\} : x \in K\}.
	\]
\end{definition}
A Banach space \( X \) is said to have weak normal structure if every weakly compact convex set \( K \subseteq X \) containing more than one point possesses normal structure. Normal structure is crucial in the fixed point theory of nonexpansive mappings, as a reflexive Banach space with normal structure guarantees the fixed point property for such mappings (see \cite{12}).

In \cite{13}, Sims introduced the concept of WORTH for Banach spaces. A Banach space $X$ is said to have the WORTH property if  
\[
\limsup_{n\to\infty} \left| \|x_n + x\| - \|x_n - x\| \right| = 0
\]  
for every weakly null sequence $\{x_n\}$ in $X$ and for all $x \in X$.

In \cite{14}, Sims defined the parameter  
\[
\omega(X) = \inf \left\{ \lambda > 0 : \lambda \liminf_{n \to \infty} \|x_n + x\| \leq \liminf_{n \to \infty} \|x_n - x\| \right\},
\]  
with the infimum taken over all weakly null sequences $\{x_n\}$ in $X$ and all $x \in X$. It was proven that $1 \leq \omega(X) \leq 3$ for any Banach space $X$.

Let $\{x_n\}$ be a bounded sequence in $X$. The asymptotic radius $r(C, \{x_n\})$ and the asymptotic center $A(C, \{x_n\})$ of $\{x_n\}$ in a subset $C \subseteq X$ are defined as follows:  
\[
r(C, \{x_n\}) = \inf \left\{\limsup_{n} \|x_n - x\| : x \in C\right\},
\]  
and  
\[
A(C, \{x_n\}) = \left\{x \in C : \limsup_{n} \|x_n - x\| = r(C, \{x_n\})\right\},
\]
respectively. It is well known that $A\left(C,\left\{{{x}_{n}} \right\}\right)$ is a nonempty weakly compact convex set whenever $C$ is. The sequence $\left\{{{x}_{n}} \right\}$ is called regular with respect to $C$ if $r\left(C,\left\{{{x}_{n}} \right\}\right)=r\left(C,\left\{x_{n_i}\right\}\right)$ for all subsequences $\left\{x_{n_i}\right\}$ of $\left\{{{x}_{n}} \right\}$. If $D$ is a bounded subset of $X$, the Chebyshev radius of $D$ relative to $C$ is defined by
$$
r_C(D)=\inf _{x_1 \in C} \sup _{x_2 \in D}\|x_1-x_2\|.
$$
The (DL)-condition is defined as follows:
\begin{definition}
If there exists a $\kappa \in [0, 1)$ such that for any weakly compact convex subset $C \subseteq X$ and any bounded sequence $\{x_n\}$ in $C$ that is regular with respect to $C$, we have  
\[
r_C(A(C, \{x_n\})) \leq \kappa r(C, \{x_n\}),
\]  
where $r_C$ denotes the radius within $C$.
\end{definition}
In \cite{15}(Theorem 3), we can get the following lemma. 
\begin{Lemma}\label{l10}
Let $a\in C$, then for every $\varepsilon>0$, there exists $M \in \mathbb{N}$ such that
	
	(i) $\left\|x_M-a\right\| \leq r\left(C,\left\{{{x}_{n}} \right\}\right)+\varepsilon$,
	
	(ii) $\left\|x_M-2 x+a\right\|\leq  \omega(r\left(C,\left\{{{x}_{n}} \right\}\right)+\varepsilon)$,
	
	(iii) $\left\|x_M-\left(\frac{2}{\omega^2+1} x+\frac{\omega^2-1}{\omega^2+1} a\right)\right\| \geq r\left(C,\left\{{{x}_{n}} \right\}\right)-\varepsilon$,
	
	(iv)
	$\left\|\left(\omega^2-1\right)\left(x_M -x\right)-\left(\omega^2+1\right)(a-x)\right\| \geq\left(\omega^2+1\right)\|a-x\|\left(\frac{r\left(C,\left\{{{x}_{n}} \right\}\right)-\varepsilon}{r\left(C,\left\{{{x}_{n}} \right\}\right)}\right) .$
\end{Lemma}
\begin{Theorem}\label{t5}
	Let $C$ be a weakly compact convex subset of a Banach space $X$, and let $\left\{{{x}_{n}} \right\}$ be a bounded sequence in $C$, regular with respect to $C$, then
	$$
 r_C\left(A\left(C,\left\{{{x}_{n}} \right\}\right)\right) \leq 
		 \frac{\omega(X)\bigg([2^{p-3}(\lambda^p+\mu^p)C^p_{- \infty }(\lambda,\mu, X)(\omega(X)^p+1)]^{\frac1p}+(\lambda-\mu)\bigg)}{\omega(X)^2+1} r\left(C,\left\{{{x}_{n}} \right\}\right).
	$$
\end{Theorem} 
\begin{proof}
Let $a\in C$, from now on, to make the proof easier, we denote $r=r\left(C,\left\{{{x}_{n}} \right\}\right)$, $A=A\left(C,\left\{{{x}_{n}} \right\}\right)$ and $\omega=\omega(X)$. Put $u=\omega^2\left(x_M -a\right)$ and $v=x_M -2 x+a$, by the above Lemma \ref{l10}, we have $\|u\| \leq \omega^2(r+\varepsilon),\|v\| \leq \omega(r+\varepsilon)$. Then we get$$\begin{aligned}
		\| \lambda u+\mu v \| 
		&\geq\lambda\|u+v\|-|\lambda-\mu|\|v\|
		\\&\geq\lambda\left\|\omega^2\left(\left(x_M -x\right)-(a-x)\right)+\left(x_M -x\right)+(a-x)\right\|-|\lambda-\mu|\omega(r+\varepsilon) \\&=\lambda\left(\omega^2+1\right)\left\|\left(x_M -x\right)-\frac{\omega^2-1}{\omega^2+1}(a-x)\right\|-|\lambda-\mu|\omega(r+\varepsilon) \\
		& =\lambda\left(\omega^2+1\right)\left\|x_M -\left(\frac{2}{\omega^2+1} x+\frac{\omega^2-1}{\omega^2+1} a\right)\right\|-|\lambda-\mu|\omega(r+\varepsilon) \\
		& \geq\lambda\left(\omega^2+1\right)(r-\varepsilon)-|\lambda-\mu|\omega(r+\varepsilon),
	\end{aligned}$$
and	$$\begin{aligned}
		\|\mu u-\lambda v\| &\geq\mu\|u-v\|-|\lambda-\mu|\|v\| \\&\geq\mu\left\|\omega^2\left(\left(x_M -x\right)-(a-x)\right)-\left(x_M -x\right)-(a-x)\right\|- |\lambda-\mu|\omega(r+\varepsilon)\\
		& =\mu\left\|\left(\omega^2-1\right)\left(x_M -x\right)-\left(\omega^2+1\right)(a-x)\right\|-|\lambda-\mu|\omega(r+\varepsilon) \\
		& \geq\mu\left(\omega^2+1\right)\|a-x\|\left(\frac{r-\varepsilon}{r}\right)-|\lambda-\mu|\omega(r+\varepsilon).
	\end{aligned}$$Thus we obtain that$$\begin{aligned}
	C^p_{- \infty }(\lambda,\mu, X)&\geq\frac{\min\{\| \lambda u + \mu v \|^p , \| \mu u - \lambda v \|^p\}}{2^{p-3} (\lambda^p + \mu^p) (\| u \|^p + \| v \|^p)}\\&\geq\frac{\min\{[\lambda\left(\omega^2+1\right)(r-\varepsilon)-|\lambda-\mu|\omega(r+\varepsilon)]^p,[\mu\left(\omega^2+1\right)\|a-x\|\left(\frac{r-\varepsilon}{r}\right)-|\lambda-\mu|\omega(r+\varepsilon)]^p\}}{2^{p-3}(\lambda^p+\mu^p)\omega^p(\omega^p+1)(r+\varepsilon)^p}.
	\end{aligned} $$
		Then since $\|a-x\|\leq r$ and let $\varepsilon\to 0$, we obtain that $$C^p_{- \infty }(\lambda,\mu, X)\geq\frac{[\min\{\lambda,\mu\}\left(\omega^2+1\right)\|a-x\|-|\lambda-\mu|\omega r]^p}{2^{p-3}(\lambda^p+\mu^p)\omega^p(\omega^p+1)r^p},$$
		which implies that $$\|a-x\|\leq\frac{\omega\bigg([2^{p-3}(\lambda^p+\mu^p)C^p_{- \infty }(\lambda,\mu, X)(\omega^p+1)]^{\frac1p}+|\lambda-\mu|\bigg)r}{\min\{\lambda,\mu\}(\omega^2+1)}.$$  Hence $$r_C(A)\leq\frac{\omega\bigg([2^{p-3}(\lambda^p+\mu^p)C^p_{- \infty }(\lambda,\mu, X)(\omega^p+1)]^{\frac1p}+|\lambda-\mu|\bigg)}{\min\{\lambda,\mu\}(\omega^2+1)}r$$
		 holds for arbitrary $a\in C $.
\end{proof}
Next we will use the following two lemmas to obtain two corollaries based on this theorem.
\begin{Lemma}\label{l4}
	Let $X$ be a Banach space satisfying the
	(DL)-condition, then $X$ has weak normal structure.
\end{Lemma}
\begin{Lemma}\label{l5}
	Let $C$ be a nonempty weakly compact convex subset of a Banach space $X$ which satisfies the (DL)-condition, and let $T: C \rightarrow C$ be a multivalued nonexpansive mapping, then $T$ has a fixed point.
\end{Lemma}
\begin{Corollary}
Let $C$ be a nonempty bounded closed convex subset of a Banach space $X$ such that $$C^p_{- \infty }(\lambda,\mu, X)<\frac{[\min\{\lambda,
	\mu\}\left(\omega(X)^2+1\right)-|\lambda-\mu|\omega(X) ]^p}{2^{p-3}(\lambda^p+\mu^p)\omega(X)^p(\omega(X)^p+1)},$$ and let $T: C \rightarrow K C(C)$ be a multivalued nonexpansive mapping, then $T$ has a fixed point.
\end{Corollary}

\begin{proof}
	From the condition
$$C^p_{- \infty }(\lambda,\mu, X)<\frac{[\min\{\lambda,
	\mu\}\left(\omega(X)^2+1\right)-|\lambda-\mu|\omega(X) ]^p}{2^{p-3}(\lambda^p+\mu^p)\omega(X)^p(\omega(X)^p+1)},$$
	then $X$ satisfies the (DL)-condition by Theorem \ref{t5}, therefore $T$ has a fixed point by Lemma \ref{l5}.
\end{proof} 
\begin{Corollary}
	 Let $X$ be a Banach space such that $$C^p_{- \infty }(\lambda,\mu, X)<\frac{[\min\{\lambda,
	 	\mu\}\left(\omega(X)^2+1\right)-|\lambda-\mu|\omega(X) ]^p}{2^{p-3}(\lambda^p+\mu^p)\omega(X)^p(\omega(X)^p+1)},$$ then $X$ has normal structure.
\end{Corollary}
\begin{proof}
First, using Theorem \ref{t5} and Lemma \ref{l4}, we can easily show that $X$ has weak normal structure.

Second, we have the inequality
\[
C^p_{- \infty }(\lambda,\mu, X) < \frac{[\min\{\lambda, \mu\}(\omega(X)^2 + 1) - |\lambda - \mu|\omega(X)]^p}{2^{p-3}(\lambda^p + \mu^p)\omega(X)^p(\omega(X)^p + 1)} < \frac{(\lambda + \mu)^p}{2^{p-2}(\lambda^p + \mu^p)},
\]
where $1 \leq \omega(X) \leq 3$. This implies that $X$ is uniformly non-square by Corollary \ref{c1}, and thus $X$ is reflexive. Consequently, weak normal structure is equivalent to normal structure in $X$.
\end{proof}
\mbox{}
\\[8pt]
{\bf Acknowledgements}\ \ We thank the referees for their time and comments.
\\[8pt]
{\bf Conflict of Interest}\ \ The authors declare no conflict of interest.

\end{document}